\documentclass[a4paper]{article}
\usepackage{graphicx,amssymb,amsfonts,amsmath,multirow,hyperref,cite,amsthm}
\usepackage[a4paper,includeheadfoot,margin=1in]{geometry}

\graphicspath{ {./piccies/} }

\usepackage{float}

\newtheorem{theorem}{Theorem}
\numberwithin{theorem}{section}
\newtheorem{corollary}[theorem]{Corollary}
\newtheorem{observation}[theorem]{Observation}
\newtheorem{lemma}[theorem]{Lemma}

\title{Effective Computation of the Heegaard Genus of 3-Manifolds}

\author{
  Benjamin A. Burton$^1$\\
  \and
  Finn Thompson$^2$
}

\date{
	{\itshape\small{$^{1,2}$School of Mathematics and Physics, University of Queensland, Brisbane QLD 4072, Australia.}}\\
	{\itshape\small{$^1$}\ttfamily\small{bab@maths.uq.edu.au}\hspace{2cm}\itshape\small{$^2$}\ttfamily\small{f.thompson@uq.net.au}}\\[2ex]%
	\today
}

\begin{document}

\maketitle

\begin{abstract}

The Heegaard genus is a fundamental invariant of 3-manifolds. However, computing the Heegaard genus of a triangulated 3-manifold is NP-hard, and while algorithms exist, little work has been done in making such an algorithm efficient and practical for implementation. Current algorithms use almost normal surfaces, which are an extension of the algorithm-friendly normal surface theory but which add considerable complexity for both running time and implementation.

Here we take a different approach: instead of working with almost normal surfaces, we give a general method of modifying the input triangulation that allows us to avoid almost normal surfaces entirely. The cost is just four new tetrahedra, and the benefit is that important surfaces that were once almost normal can be moved to the simpler setting of normal surfaces in the new triangulation. We apply this technique to the computation of Heegaard genus, where we develop algorithms and heuristics that prove successful in practice when applied to a data set of 3,000 closed hyperbolic 3-manifolds; we precisely determine the genus for at least 2,705 of these.

\end{abstract}


\paragraph{Supplementary Material}Software (Source Code): \href{https://github.com/FinnThompson99/heegaard}{github.com/FinnThompson99/heegaard}
	
\paragraph{Acknowledgements}We thank the referees for their helpful comments. The second author was partially supported by the Australian Research Council (grant DP150104108).

\section{Introduction}

In topology, invariants are used to distinguish between different manifolds. Powerful invariants are often hard to utilise computationally, and sometimes only partial information can be calculated for them. One such example is the fundamental group, which uniquely determines hyperbolic 3-manifolds \cite{MR236383}, but it is non-trivial to determine if two such groups are isomorphic, given their group presentations \cite{MR1324134}.

We focus on an invariant known as the \textit{Heegaard genus} of a 3-manifold, which is the minimal genus of a surface (a Heegaard surface) that splits a given manifold into two handlebodies of equal genus. For example, the 3-sphere has Heegaard genus 0, and all lens spaces have Heegaard genus 1. The Heegaard genus can be used to calculate the tunnel number of a knot, which is a useful invariant in knot theory (\cite{MR0582900}, \cite{MR3431010}).

However, computing Heegaard genus is NP-hard \cite{MR3726594}. Algorithms of Rubinstein, Lackenby, Li, and Johannson (\cite{MR1470718}, \cite{MR2443101}, \cite{MR2821570}, \cite{MR1027902}) provide theoretical methods, but primarily argue the existence of an algorithm; these algorithms are extremely complex to describe and theoretically extremely slow to run. Furthermore, no implementations currently exist.

These three algorithms all make use of surfaces embedded in triangulations of 3-manifolds. In particular, they use normal and almost normal surfaces, which are commonly used to certify topological properties of a manifold. Normal surfaces are described by vectors in $\mathbb{Z}_{\geq0}^{7n}$ for triangulations with $n$ tetrahedra, and can be efficiently generated using high-dimensional polytope vertex enumeration methods. However, \textit{almost normal} surfaces are more complicated, requiring vectors in $\mathbb{Z}_{\geq0}^{35n}$, which can significantly affect the running time of any algorithm using them, due to an exponential dependence on the dimension of the vector space.

Our strategy in this paper is to avoid the costs of almost normal surfaces by \textit{modifying the triangulation}. Since we are solving topological problems, not combinatorial problems, we aim to retriangulate our manifolds so that important almost normal surfaces in the original triangulation (such as a surface that realises a Heegaard splitting) become \textit{normal} in the new triangulation. This then allows us to attack the problem using well-studied normal surface techniques, such as the highly optimised algorithms in the software package \textit{Regina} \cite{regina}.

Our main tool is a new \textit{gadget}, which replaces a single tetrahedron and serves to eliminate the almost normal portion of the surface. Importantly, we show that this gadget preserves the \textit{zero-efficiency} of the triangulation, a property that plays a major role in making normal surface algorithms both implementable and fast (\cite{MR3208240}, \cite{MR2057531}). We identify a class of exceptional surfaces with which this gadget cannot be used, but these exceptions are of a very specific form and are simple to analyse. For computing Heegaard genus, we show how to work around these exceptions with an alternate method for locating Heegaard splittings of this form.

Using our techniques, we are able to successfully compute the Heegaard genus of 2,705 closed, hyperbolic 3-manifolds drawn from the hyperbolic census of Hodgson and Weeks \cite{MR1341719}.

\section{Preliminaries}

\subsection{Triangulations}

To study topological objects computationally, we represent $n$-manifolds using \textbf{triangulations}. These are a pairwise gluing of $n$-dimensional simplicies along their $(n-1)$-dimensional facets. Specifically, a \textbf{3-manifold triangulation} consists of a finite number of tetrahedra that are identified (or `glued') along pairs of triangular faces. Every compact 3-manifold can be represented by a triangulation \cite{MR48805}.

Let $T$ be some arbitrary triangulation (not necessarily a simplicial complex), and consider some tetrahedron $\tau$ in $T$. In the triangulation, edges, triangles (faces) and vertices of $\tau$ are not necessarily distinct. We use a canonical numbering of the vertices, as in Figure \ref{fig:tri0123}.

\begin{figure}[H]
    \centering
    \includegraphics[width=0.2\textwidth]{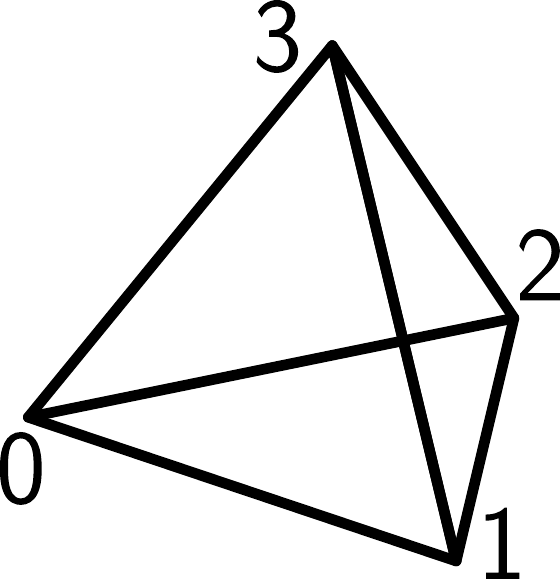}
    \caption{Canonical numbering of the vertices of a tetrahedron.}
    \label{fig:tri0123}
\end{figure}

Let $\tau^\Delta_{ijk}$ denote the triangular face constrained by vertices $(i,j,k)$ of tetrahedron $\tau$. If a given $\tau^\Delta_{ijk}$ is not a boundary facet, then this triangle is shared between (not necessarily distinct) tetrahedra $\tau$ and $\bar{\tau}$, and we may write  $\tau^\Delta_{ijk}=\bar{\tau}^\Delta_{pqr}$.

\subsection{(Almost) Normal Surfaces}
If a surface $S$ is properly embedded in a triangulation $T$, we call it a \textbf{normal surface} if it meets each tetrahedron $\tau$ of $T$ in a collection of triangular or quadrilateral \textbf{normal} discs (Figure \ref{fig:normalPieces}). In each tetrahedron, triangular discs separate one vertex from the remaining three while quadrilateral discs separate pairs of vertices. In each tetrahedron, there can be at most one of any quadrilateral type.

\begin{figure}[H]
    \centering
    \includegraphics[width=0.25\textwidth]{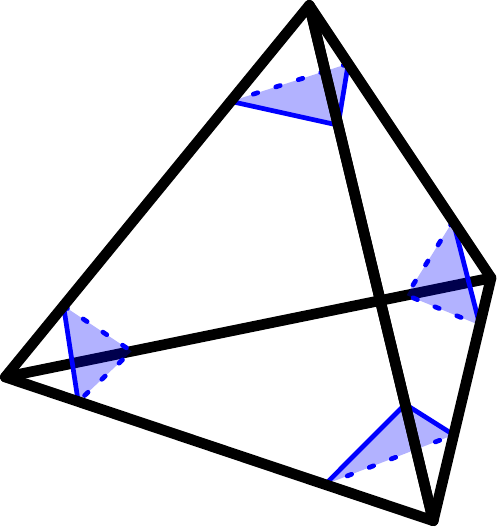}\phantom{===========}
    \includegraphics[width=0.25\textwidth]{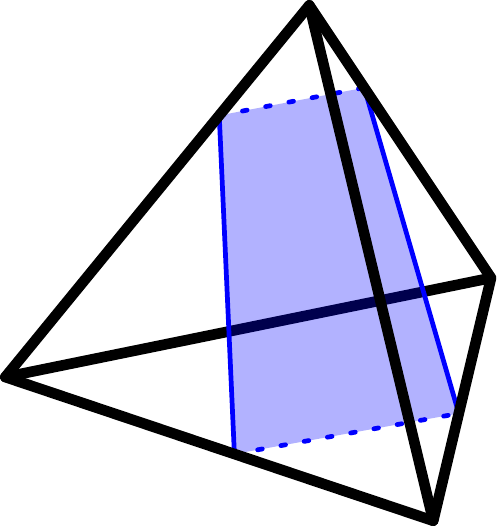}
    \caption{The four triangle pieces, and one of the three quadrilateral pieces.}
    \label{fig:normalPieces}
\end{figure}

One such surface is a \textbf{vertex linking} sphere, consisting only of triangle pieces which together form a sphere (in a closed manifold) that surrounds a particular vertex in the triangulation.

Normal surfaces are represented by \textbf{normal coordinates}, which are vectors $\mathbf{v}\in\mathbb{Z}_{\geq0}^{7|T|}$, notated by $\mathbf{v}=(v_1,v_2,\cdots,v_7,\cdots,v_{7|T|})$, where $v_i$ represent the number of each type in each tetrahedron, and $|T|$ is the number of tetrahedra in the triangulation $T$.

We call two surfaces \textbf{locally compatible} if they are able to avoid intersection in any given tetrahedron. That is, in each tetrahedron, together they use at most one type of quadrilateral piece. Locally compatible surfaces can be added using the \textbf{Haken sum}. If $U$ and $V$ are normal surfaces with normal coordinates $\mathbf{U},\mathbf{V}$, then $U+V$ has normal coordinates $\mathbf{U}+\mathbf{V}$. $U+V$ is formed by a geometric surgery (called `regular exchange'), where sections of each surface are cut and glued around curves in $U\cap V$ (as in Figure \ref{fig:hakenSum}), such that resulting pieces are normal discs (as in Figure \ref{fig:hakenNormal}).

\begin{figure}[H]
    \centering
    \includegraphics[width=0.8\textwidth]{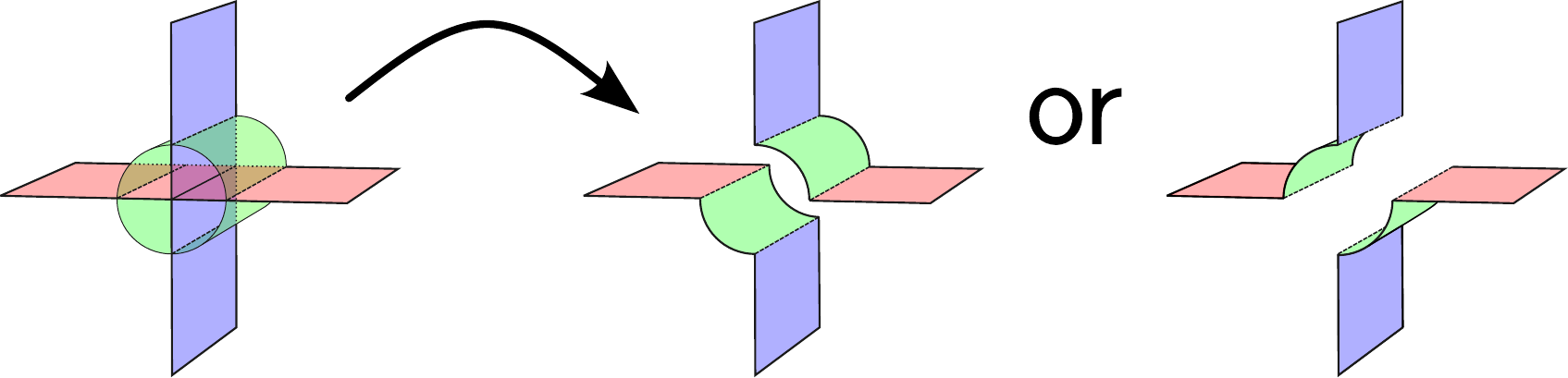}
    \caption{Around each intersection $\gamma$, the boundary of a neighbourhood (green) of $\gamma$ is partitioned between the two sheets (red, blue) in one of two ways.}
    \label{fig:hakenSum}
\end{figure}

\begin{figure}[H]
    \centering
    \includegraphics[width=0.8\textwidth]{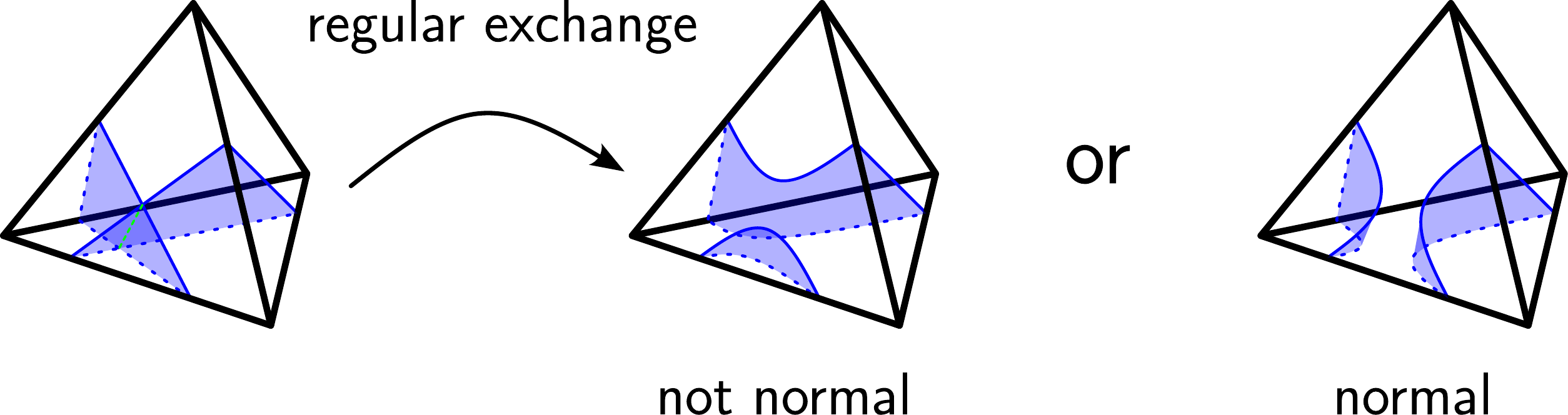}
    \caption{A regular exchange is performed on two intersecting pieces in two different ways.}
    \label{fig:hakenNormal}
\end{figure}

We call a normal surface $S$ \textbf{fundamental} if $S=A+B$ implies $\mathbf{A}=\mathbf{0}$ or $\mathbf{B}=\mathbf{0}$. By standard Hilbert basis arguments, any normal surface can be expressed as a sum of fundamental surfaces (which is possibly not unique).

An \textbf{almost normal surface} is similar to a normal surface, but with the addition that it must intersect one tetrahedron in a collection of triangles and quadrilaterals, and exactly one unknotted annulus or octagon piece (see Figure \ref{fig:almostNormalPieces}). We shall refer to annulus pieces between quadrilaterals or triangles of the same type as parallel, and all others as non-parallel.

\begin{figure}[H]
    \centering
    \includegraphics[width=0.17\textwidth]{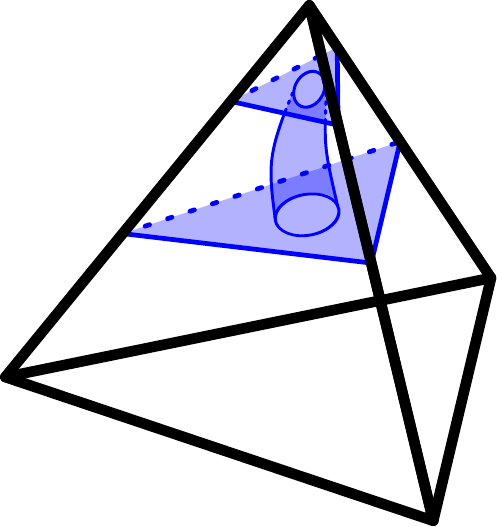}\phantom{=}
    \includegraphics[width=0.17\textwidth]{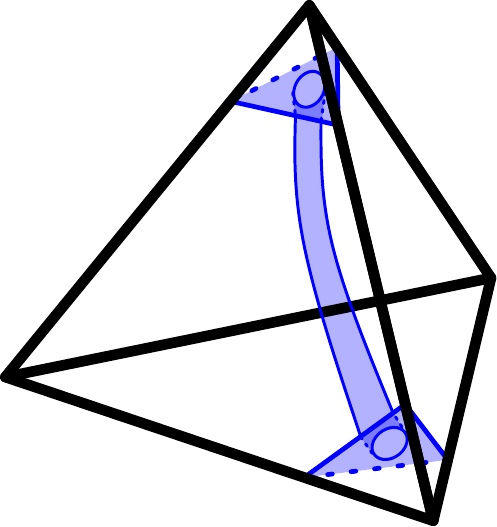}\phantom{=}
    \includegraphics[width=0.17\textwidth]{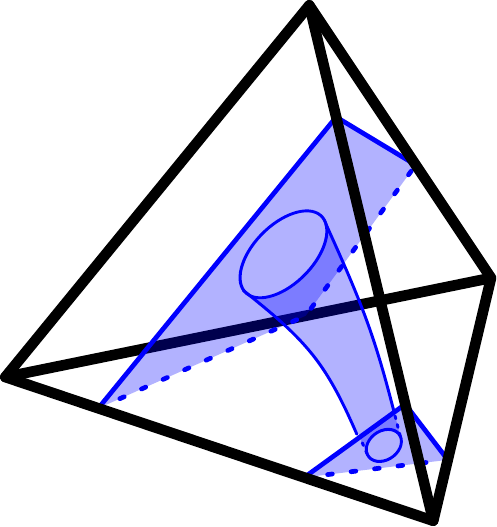}\phantom{=}
    \includegraphics[width=0.17\textwidth]{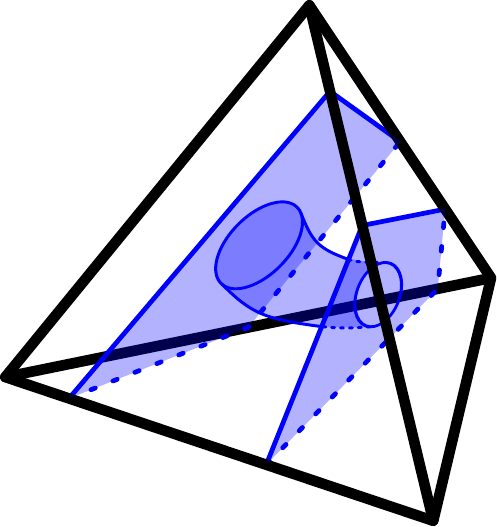}\phantom{=}
    \includegraphics[width=0.17\textwidth]{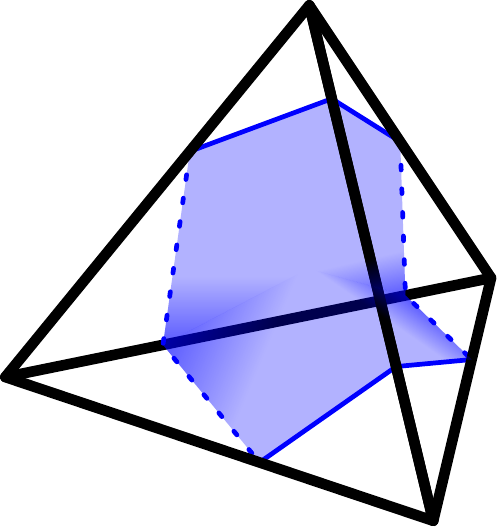}
    \caption{Annuli between parallel triangles, non-parallel triangles, a triangle and a quadrilateral, and parallel quadrilaterals; an octagon piece.}
    \label{fig:almostNormalPieces}
\end{figure}

Of these almost normal pieces, there are three octagons, three parallel quadrilateral annuli, four parallel triangle annuli, six non-parallel triangle annuli and twelve triangle-quadrilateral annuli. With the four triangle and three quadrilateral normal pieces, this means that normal coordinates for almost normal surfaces are vectors in $\mathbb{Z}_{\geq0}^{35|T|}$.

We call a triangulation of a closed 3-manifold \textbf{zero-efficient} if its only normal 2-spheres are vertex linking. Indeed, all closed, orientable, irreducible 3-manifolds (excluding $\mathbb{S}^3$, $L(3,1)$ and $\mathbb{R}P^2$) can be modified to a zero-efficient triangulation, with only one vertex. A celebrated result of Jaco and Rubinstein states that if a zero-efficient triangulation of a 3-manifold $M$ has an almost normal 2-sphere, then $M=\mathbb{S}^3$ \cite{MR2057531}.

\subsection{Additional Notation}
For a normal or almost normal surface $S$ in a triangulation $T$, let $S_U$ represent the part of $S$ restricted to some connected sub-triangulation $U\subset T$.

For a normal surface $S$ in a triangulation $T$, we use $\mathbf{S_\tau}=(t_0,t_1,t_2,t_3,q_{01/23},q_{02/13},q_{03/12})$ to refer to the normal coordinates of $S$ restricted to some tetrahedron $\tau$.

When referring to specific normal or almost normal pieces, we shall use the following.
\begin{itemize}
    \item \texttt{tri\_a}: the link of the vertex $a$
    \item \texttt{quad\_ab/cd}: the quadrilateral separating vertices $a,b$ from $c,d$
    \item \texttt{oct\_ab/cd}: the octagon separating vertices $a,b$ from $c,d$
    \item \texttt{tri\_a:tri\_b}: the annulus connecting \texttt{tri\_a} to \texttt{tri\_b}, allowing $a=b$.
    \item \texttt{tri\_e:quad\_ab/cd}: the annulus connecting \texttt{tri\_e} and \texttt{quad\_ab/cd}, where $e\in\{a,b,c,d\}$.
\end{itemize}

\subsection{Normalisation Moves}\label{appNORMAL}
We can convert surfaces which are not yet normal into normal surfaces through a process called normalisation \cite{MR2057531}. This is an iterative process using \textit{normalisation moves}, used to convert a surface into a normal surface (possibly making changes to the topology of the surface, but which are harmless).

Two particular examples that are relevant to our work are as follows:
\begin{itemize}
    \item \textbf{Compression} within a tetrahedron, wherein a tubed region of the surface is cut and then `capped' off with a disk. This does not change the intersection of the surface with the boundary of the tetrahedron, but reduces the Euler characteristic.
    \item \textbf{Pulling} a piece across the boundary of a tetrahedron, where the surface is `pulled' across an edge, reducing the number of intersections with the boundary of the tetrahedron. This possibly changes the Euler characteristic.
\end{itemize}
This second move may be used on octagon pieces in a normal surface, as seen in Figure \ref{fig:octagonCompression} where an isotopy of \texttt{oct\_02/13} results in just \texttt{tri\_0} and \texttt{tri\_2}.

\begin{figure}[H]
    \centering
    \includegraphics[width=0.7\textwidth]{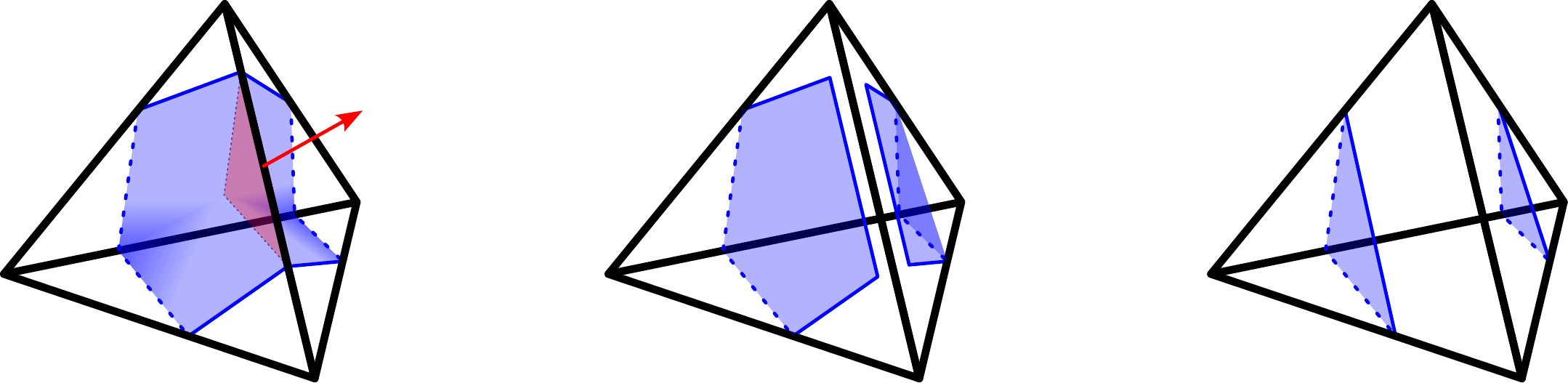}
    \caption{An octagon is `pulled' towards an edge, creating two triangles.}
    \label{fig:octagonCompression}
\end{figure}
Normalisation moves may introduce other non-normal pieces, but combinations of these moves will eventually yield a normal surface.

\subsection{Implementations of Normal and Almost Normal Surfaces}\label{runningtimeseccie}
\textit{Regina} \cite{regina} efficiently implements data structures and algorithms for normal surfaces. However, they are still inherently exponential methods. Indeed, we have found an experimental bound on the running time of fundamental normal surface enumeration for closed hyperbolic 3-manifolds, $T(n)\in\mathcal{O}(1.934^n)$ for a triangulation with $n$ tetrahedra. See \cite{burton14-hilbert-alenex} for more rigorous experimental results. The number of fundamental normal surfaces for a triangulation with $n$ tetrahedra is $C(n)\in\mathcal{O}(\phi^n)$ where $\phi=\frac{1+\sqrt{5}}{2}$ \cite{MR2742952}. On the same set of closed hyperbolic 3-manifolds, we have experimentally found $C(n)\in\mathcal{O}(1.503^n)$. See Appendix \ref{appTIMES} for details. We use these experimental recordings in place of more theoretical bounds, as theoretical bounds generally represent the worst-case running times.

Almost normal surfaces have historically been motivated by 3-sphere recognition \cite{MR1403961}, which by a result of Thompson only requires use of the octagon piece \cite{MR1295555}.
In this special setting, there are techniques to normalise an almost normal surface with an octagon piece---for instance, \textit{Regina} features a simple algorithm to remove octagons from an almost normal surface by modifying the triangulation using a 3-tetrahedron gadget. This technique is used in the algorithm for cutting along almost normal surfaces \cite{githubIssue}.

The more general setting that allows octagon \emph{or} annulus pieces is significantly more complex to work with (with normal coordinates in $\mathbb{Z}^{35|T|}$), and there are no implementations known to the authors.

\subsection{Heegaard Genus}\label{heegen}

Any closed, orientable 3-manifold $M$ can be represented as $M=H_1\cup H_2$ where $H_1$ and $H_2$ are \textbf{handlebodies} (solid genus $n$ tori) with $\partial H_1=\partial H_2=H_1\cap H_2$. This decomposition is called a \textbf{Heegaard splitting} of $M$, and $\partial H_1$ is the \textbf{Heegaard surface}. The \textbf{Heegaard genus} of $M$ is the minimal genus of all Heegaard splittings of $M$.

All 3-manifolds with Heegaard genus 1 are classified as being either $\mathbb{S}^1\times \mathbb{S}^2$ or a Lens space. Similarly, $\mathbb{S}^3$ is the only 3-manifold with Heegaard genus 0 \cite{MR2893651}. Manifolds with genus $\geq2$ are not fully classified, and only some classes of examples are known. In general, computing the Heegaard genus of a 3-manifold is an NP-hard problem \cite{MR3726594}.

We choose to follow the techniques of Rubinstein for determining Heegaard genus \cite{MR1470718}. For a zero-efficient triangulation $T$ of a closed, orientable $3$-manifold, a Heegaard splitting will be of the form $S=S_{an}+S_t$ where $S_{an}$ is an almost normal surface with negative Euler characteristic, and $S_t$ is a normal surface with zero Euler characteristic \cite{MR1470718}.

To determine if a triangulation has a genus $g$ Heegaard splitting, we generate candidate almost normal surfaces and test all combinations for $S$ with Euler characteristic $\chi(S_{an})=2-2g<0$. That is, we cut along each candidate surface, and test whether the resulting piece/s are two handlebodies of genus $g$, using an efficient algorithm for handlebody recognition \cite{MR4604006}. We use algebraic properties of a given manifold $M$ to form a lower bound on genus \cite{MR3037787}. In particular, we have $\mathrm{rank}(H_1(M))\leq\mathrm{rank}(\pi_1(M))\leq g$ where $\mathrm{rank}(H_1(M))$ and $\mathrm{rank}(\pi_1(M))$ are the number of generators of the first homology group and fundamental group of $M$, respectively. Hence, we use $\max(2,\,\mathrm{rank}(H_1(M)))\leq g$ as a lower bound.

\section{A Gadget for Normalising Almost Normal Surfaces}

Given an orientable triangulation $T$ and an almost normal surface $S$ in $T$, we seek a way to modify $T$ into another triangulation $T'$ of the same 3-manifold, so that $S$ can be expressed in $T'$ using normal surfaces only, similar to the octagon gadget (see Section \ref{runningtimeseccie}). As almost normal pieces only appear within a single tetrahedron, we desire a small triangulation that we may substitute for a tetrahedron. That is, we desire an orientable triangulation $G$ that is homeomorphic to the 3-ball ($B^3$), with four vertices and four boundary faces. We shall call such a triangulation \textbf{the gadget}.

\subsection{A General Gadget Construction}

Suppose that $G$ is a triangulation such that $G\simeq B^3$, with four vertices and four boundary faces, arranged to form the boundary of a tetrahedron. We define $\partial G^\Delta_{ijk}$ to represent the boundary facet that is constrained by vertices $(i,j,k)$ of the gadget $G$. Then, the boundary of $G$ is isomorphic to the boundary of a single tetrahedron, and so we can effectively `replace' a tetrahedron $\tau$ in $T$ with $G$ as follows.

We define $T_\tau^{abcd(G)}$ to be the resulting triangulation formed by removing $\tau$ from $T$, and gluing $G$ in its place according to the permutation mapping vertices of $\tau$ to $G$ by $(0,1,2,3)\mapsto(a,b,c,d)$. This operation is simply un-gluing each face of $\tau$ from $T$, then gluing the corresponding boundary face of $G$ in its place (as indicated in Figure \ref{fig:gadgetGluing}) where
\[\tau^\Delta_{012}\curvearrowright \partial G^\Delta_{abc},\quad \tau^\Delta_{013}\curvearrowright \partial G^\Delta_{abd},\quad\tau^\Delta_{023}\curvearrowright \partial G^\Delta_{acd},\quad\tau^\Delta_{123}\curvearrowright \partial G^\Delta_{bcd}.\]

\begin{figure}[H]
    \centering
    \includegraphics[scale=0.2]{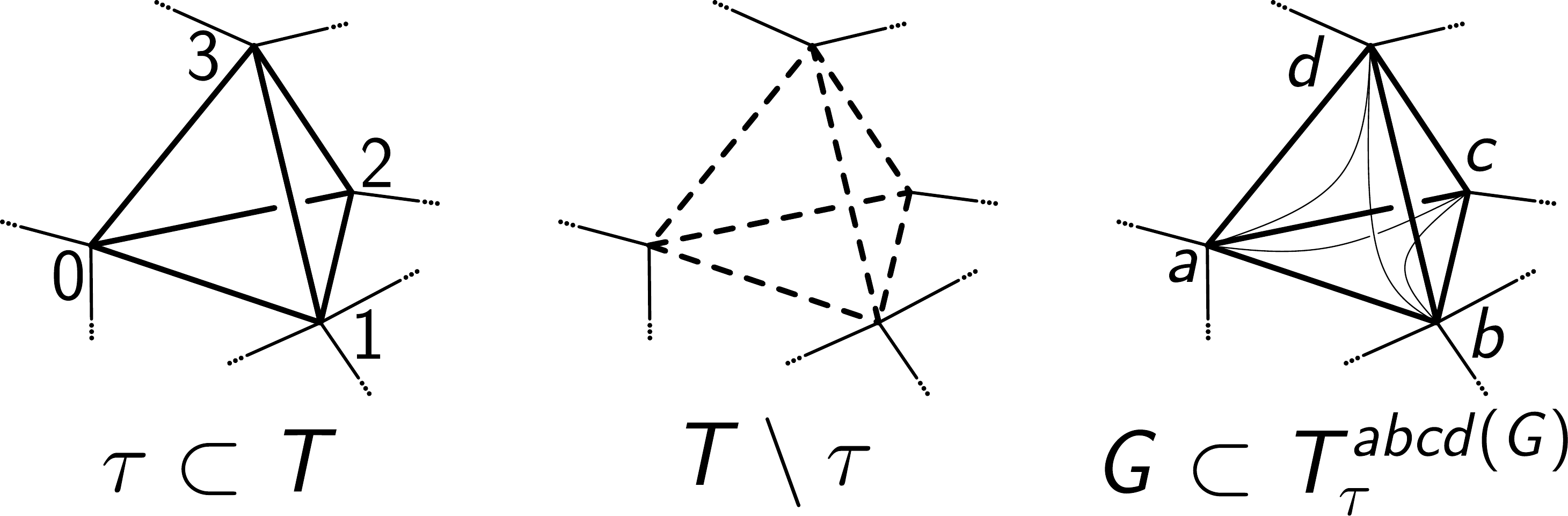}
    \caption{Some tetrahedron $\tau$ in $T$; $T$ with $\tau$ removed; an abstract depiction of $G$ in $T'=T_\tau^{abcd(G)}$.}
    \label{fig:gadgetGluing}
\end{figure}

If any pair of faces of $\tau$ are glued to each other, for example $\tau^\Delta_{012}=\tau^\Delta_{312}$, then we glue the boundary faces of the gadget accordingly, such as $\partial G^\Delta_{abd}\curvearrowright\partial G^\Delta_{dab}$.

\subsection{Gadget Requirements}\label{reqs}
Here, we discuss motivations behind a choice for the gadget. As discussed, we assume that $G$ has isomorphic boundary to the boundary of a tetrahedron, and is itself homeomorphic to $B^3$. Now, consider some gluing of the gadget into a triangulation to form $T'=T_\tau^{abcd(G)}$. 
As $\partial G\cong\partial\tau$ and $G\cong\tau$, $G$ does not modify topological properties of the triangulation, $T$, in which it is placed. In addition, we desire the following.

\begin{itemize}
    \item If $T$ is zero-efficient, then $T'=T_\tau^{abcd(G)}$ is also zero-efficient, aside from simple exceptions.
    \item There must be a well-defined choice of normal surfaces in $G$ that each correspond to one of the four triangular or three quadrilateral pieces of a normal surface in $\tau$. Similarly, there must be a well-defined choice of an octagon piece, and some annulus pieces (discussed in Lemma \ref{parallelLemma}), such that local compatibility of these surfaces matches the local compatibility of (almost) normal pieces. For example, we require that the triangle pieces in $G$ are compatible with all other pieces.
    
    This means that if $S$ is a normal surface in $T$, then there must exist some normal surface $S'$ in $T'$ with identical normal coordinates $\mathbf{S_{T\setminus\tau}}=\mathbf{S'_{T'\setminus G}}$ outside of $\tau$/$G$, where additionally $S$ is isotopic to $S'$. Similarly, the normal surfaces of $G$ chosen to represent (almost) normal pieces should be sufficient to construct any almost normal surface of $T$, aside from some manageable exceptions. That is, if a surface $S$ has an almost normal piece in $\tau$, then for some (possibly distinct) tetrahedron $\tau'$ and some permutation $(a,b,c,d)$, there must exist a normal surface $S'$ in $T_{\tau'}^{abcd(G)}$ where $S$ is isotopic to $S'$. 
\end{itemize}

\subsection{Choice of the Gadget}

The requirements that the boundary of $G$ must be isomorphic to the boundary of a tetrahedron, and that $G$ is homeomorphic to $B^3$, are both targeted using \textit{Regina}'s \texttt{tricensus} command, which generates a list of triangulations satisfying given conditions.

To ensure that the chosen triangulation would have surfaces equivalent to annulus or octagon pieces, a search of the fundamental normal surfaces was conducted. Examining the boundary components of each surface for each triangulation yielded a candidate for our gadget $G$.

We now denote the 5-tetrahedron triangulation with isomorphism signature (a compact code used to uniquely generate triangulations in \textit{Regina}) \texttt{`fHLMabddeaaaa'} by $G$, or \textit{the gadget}.

In version 7.x of \textit{Regina}, the command \texttt{Triangulation3(`fHLMabddeaaaa')} can be used to reconstruct our gadget. Table \ref{tabGlue} provides the gluing instructions for pairs of faces of the five tetrahedra. Here, `Face $abc$' refers to the triangular face that meets vertices $(a,b,c)$, and $k (abc)$ represents face $abc$ of tetrahedron $k$.

\begin{table}[H]
\centering
\begin{tabular}{|l|l|l|l|l|}
\hline
Tetrahedron & Face 012 & Face 013 & Face 023 & Face 123 \\ \hline
0           & 2 (012)  & 1 (013)  &          & 1 (123) \\ \hline
1           & 3 (012)  & 0 (013)  & 3 (023)  & 0 (123) \\ \hline
2           & 0 (012)  & 4 (013)  & 4 (023)  & 3 (123) \\ \hline
3           & 1 (012)  &          & 1 (023)  & 2 (123) \\ \hline
4           &          & 2 (013)  & 2 (023)  &         \\ \hline
\end{tabular}
\caption{Construction of the gadget.}
\label{tabGlue}
\end{table}

Alternatively, Figure \ref{fig:gadget-pic} provides a geometric description of the construction of the gadget.

\begin{figure}[H]
    \centering
    \includegraphics[width=0.8\textwidth]{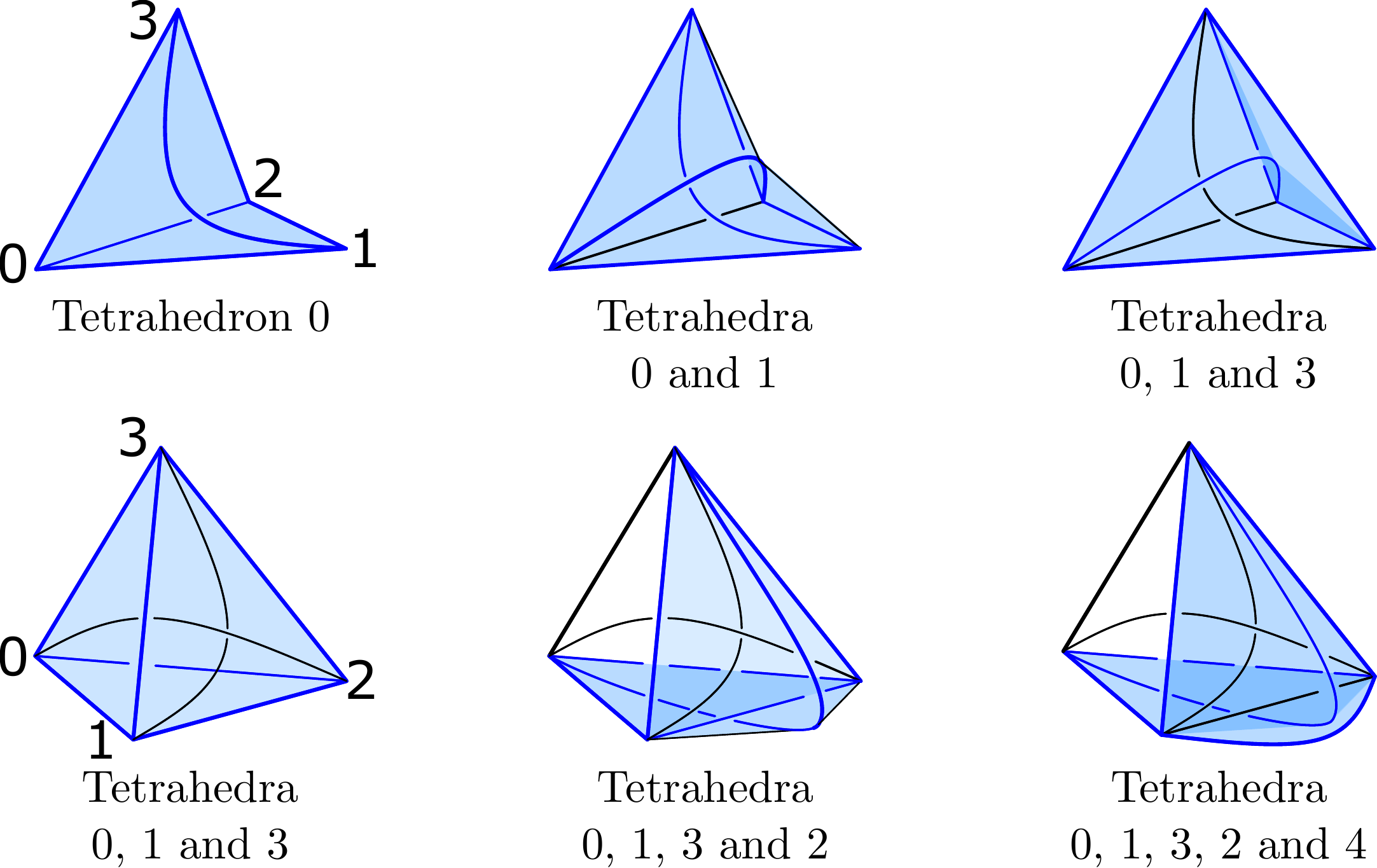}
    \caption{Geometric construction of $G$; an additional tetrahedron is glued in each step.}
    \label{fig:gadget-pic}
\end{figure}

The fundamental normal surfaces of $G$ are detailed below, along with their index when generated by \textit{Regina}. See Table \ref{tab:surfCoords} of Appendix \ref{appSURFTABLE} for details of their normal coordinates. Boundary types are determined by observing the polygonal boundary of each surface in the gadget, and determining its shape (e.g triangular boundaries have three sides).

\begin{table}[H]
\begin{tabular}{|r|r|l|l|l|}
\hline
\multicolumn{1}{|l|}{\# bdry} & \multicolumn{1}{l|}{Euler char.} & Shape                       & Boundary types & Index           \\ \hline
\multirow{4}{*}{1}               & -1                               & One-punctured torus               & Quad           & 0, 3            \\ \cline{2-5} 
                                 & \multirow{3}{*}{1}               & Octagon                           & Oct            & 17, 18          \\ \cline{3-5} 
                                 &                                  & Triangle                          & Tri            & 7, 8, 10, 13    \\ \cline{3-5} 
                                 &                                  & Quadrilateral                     & Quad           & 2, 5, 6, 11, 12 \\ \hline
\multirow{3}{*}{2}               & \multirow{3}{*}{0}               & Annulus (non-parallel triangles)          & Two tri        & 1, 4, 16, 19    \\ \cline{3-5} 
                                 &                                  & Annulus (quadrilateral, triangle) & Quad, tri      & 9, 27           \\ \cline{3-5} 
                                 &                                  & Annulus                           & Oct, tri       & 15, 26          \\ \hline
\multirow{4}{*}{3}               & \multirow{3}{*}{-1}              & Three-punctured sphere            & Three tri      & 14, 25          \\ \cline{3-5} 
                                 &                                  & Three-punctured sphere            & Quad, two tri  & 23, 24, 28      \\ \cline{3-5} 
                                 &                                  & Three-punctured sphere            & Oct, two tri   & 22              \\ \cline{2-5} 
                                 & -3                               & Three-punctured torus             & Quad, two tri  & 20              \\ \hline
4                                & -2                               & Four-punctured sphere             & Four tri       & 21              \\ \hline
\end{tabular}
\caption{Fundamental normal surfaces of $G$, detailed by number of boundaries (\# bdry), Euler characteristic, topological description, boundary types and their index.}
\label{tab:surfDetails}
\end{table}
We refer to the surface with index \texttt{i} by \texttt{sf\_i}. Note that rows where multiple indexes are listed simply represent instances of different orientations of the same surface.

Next, we explore the surfaces in $G$ that correspond to normal and almost normal pieces in a single tetrahedron $\tau$.

\paragraph*{Triangles}
The four triangular disc surfaces (\texttt{sf\_7}, \texttt{sf\_8}, \texttt{sf\_10}, \texttt{sf\_13}) are vertex links in $G$, and are hence composed of triangle pieces only. These are isotopic to the four triangular pieces we consider for a normal surface. Due to this, we introduce the following more insightful names,
\[G_{t_0}=\texttt{sf\_13},\quad G_{t_1}=\texttt{sf\_10},\quad G_{t_2}=\texttt{sf\_8},\quad G_{t_3}=\texttt{sf\_7}.\]

\paragraph*{Quadrilaterals}
The three quadrilateral pieces in $\tau$ are each isotopic to at least one of the quadrilateral surfaces in the gadget (\texttt{sf\_2}, \texttt{sf\_5}, \texttt{sf\_6}, \texttt{sf\_11}, \texttt{sf\_12}).

Surfaces \texttt{sf\_2} and \texttt{sf\_12} both represent \texttt{quad\_03/12}. However, surface \texttt{sf\_12} is not locally compatible with the annulus \texttt{tri\_3:quad\_03/12} (that is, \texttt{sf\_9}) -- this is undesired, so we choose surface \texttt{sf\_2} to represent this quadrilateral.

Similarly, surfaces \texttt{sf\_5} and \texttt{sf\_11} both represent \texttt{quad\_02/13}. Without loss of generality, we arbitrarily choose surface \texttt{sf\_5} to represent this quadrilateral. So, we introduce the following names,
\[G_{q_{01/23}}=\texttt{sf\_6},\quad G_{q_{02/13}}=\texttt{sf\_2},\quad G_{q_{03/12}}=\texttt{sf\_5}.\]

\begin{corollary}\label{normalCoords}
    With this construction, for a normal surface $S$ in a triangulation $T$, we can choose a normal surface $S'$ in $T'=T_{\tau}^{abcd(G)}$ where in $G\subset T'$, $S'_G$ has normal coordinates
    \[\mathbf{S'_{G}}=t_0\mathbf{G_{t_a}}+t_1\mathbf{G_{t_b}}+t_2\mathbf{G_{t_c}}+t_3\mathbf{G_{t_d}}+q_{01/23}\mathbf{G_{q_{ab/cd}}}+q_{02/13}\mathbf{G_{q_{ac/bd}}}+q_{03/12}\mathbf{G_{q_{ad/bc}}},\]
    where $\mathbf{S_\tau}=(t_0,t_1,t_2,t_3,q_{01/23},q_{02/13},q_{03/12})$ are the normal coordinates of $S$ in $\tau$. Hence, $S$ is isotopic to $S'$.
\end{corollary}

\paragraph*{Annuli and Octagons}
Among the normal surfaces in the gadget, we also have candidates for almost normal pieces. Namely, \texttt{sf\_17} and \texttt{sf\_18} are both discs with octagonal boundary, and we declare that
\[G_{O_{02/13}}=\texttt{sf\_18},\quad G_{O_{03/12}}=\texttt{sf\_17}.\]
Finally, with the presence of four different annulus pieces between non-parallel triangles, and two of triangle-quadrilateral type, we choose
\[G_{t_0}^{t_2}=\texttt{sf\_19},\,\,\, G_{t_0}^{t_3}=\texttt{sf\_16},\,\,\,G_{t_1}^{t_2}=\texttt{sf\_1},\,\,\,G_{t_1}^{t_3}=\texttt{sf\_4},\,\,\,G_{t_0}^{q_{03/12}}=\texttt{sf\_27},\,\,\,G_{t_3}^{q_{03/12}}=\texttt{sf\_9}.\]
Note that parallel annuli are missing from $G$, this will be discussed later.

\subsection{Verifying Requirements}

With the normal and almost normal pieces of the gadget identified, we now show that our choice of $G$ satisfies the properties discussed in Section \ref{reqs}.

\begin{observation}\label{obs}
    The sum of any two fundamental normal surfaces in $G$ can be expressed as the sum of pairwise disjoint fundamental surfaces. This can be verified in \textit{Regina}.

    Note that this is not true for all triangulations, as in general the connected components of a normal surface $A+B$ may include non-fundamental surfaces.
\end{observation}

\begin{lemma}\label{fundLemma}
    Any normal surface in the gadget can be expressed as a sum of pairwise disjoint fundamental normal surfaces.

\end{lemma}
The approach we take is to express any normal surface as the sum of connected fundamental surfaces, by repeatedly using Observation \ref{obs} on pairs of intersecting surfaces, allowing us to rewrite this surface in a way that has fewer intersections between its summands.
\begin{proof}

    Let $U=\sum_{i=1}^k U_i$ be a normal surface in the gadget, where each $U_i$ is a fundamental normal surface (not necessarily distinct) and consider the set $\Omega=\{U_i:i=1,\cdots,k\}$. Let $I$ represent the total number of intersection curves between these surfaces, noting that surfaces intersect transversely. In particular, we may `push' components of $U$ so that no two intersection curves meet within $G$.

    Now, take two surfaces $A,B\in\Omega$ that are not disjoint. By Observation \ref{obs}, in the Haken sum of $A$ and $B$, we find $A+B=\sum S_i$ where $S_i$ are pairwise disjoint fundamental normal surfaces. Consider some curve $\gamma$ in $A\cap B$. Within a small open regular neighbourhood of $\gamma$, $N(\gamma)$, a regular exchange is performed as in Figure \ref{fig:hakenSum}, such that all intersections resolve to normal discs, as in Figure \ref{fig:hakenNormal}. This decreases the number of intersections $I$ by $1$ for each $\gamma$.

    We have verified that resolving the intersections of $A+B$ does not increase $I$. Now, we must check that any intersections with any additional surfaces are not created.
    
    Take some surface $C\in\Omega$ that intersects $A$ and/or $B$. Each curve in $A\cap C$ or $B\cap C$ is outside of the regions $\bigcup_{\gamma_i\in A\cap B}N(\gamma_i)$, so evaluating $A+B$ does not create any new intersections between $C$ and any component $S_i$. In fact, the sets of intersection curves $(A\cap C)\cup (B\cap C)=\bigcup_{S_i \in A+B} (S_i \cap C)$ are equal. So, $I$ is strictly decreasing.
  
    We may continue this process again for some new pair of intersecting, fundamental surfaces, until $I=0$ and we have a collection of disjoint, fundamental surfaces.
\end{proof}

\begin{theorem}
    Let $T$ be a zero-efficient triangulation of a 3-manifold which is not $\mathbb{S}^3$, $\mathbb{R}P^3$ or $L(3,1)$. Remove some tetrahedron $\tau$ from $T$ and replace it with $G$ according to some permutation $(a,b,c,d)$. Then the resulting triangulation $T':=T_\tau^{abcd(G)}$ is also zero-efficient.%
\end{theorem}
\begin{proof}

    According to Proposition 5.1 of \cite{MR2057531}, because $T$ is zero-efficient and not $\mathbb{S}^3$, $\mathbb{R}P^3$ or $L(3,1)$, it is a one-vertex triangulation. As $\tau$ and $G$ have the same number of vertices, $T'$ is a one-vertex triangulation.
    
    So, we suppose that $T$ is zero-efficient, but $T'$ is not. That is, suppose that $T'$ has some non-vertex-linking normal 2-sphere $S$. By Lemma \ref{fundLemma}, the connected components of $S_G$ are all fundamental normal surfaces as described in Table \ref{tab:surfDetails} -- excluding \texttt{sf\_0}, \texttt{sf\_3}, \texttt{sf\_20}, which have genus $1$. Now, consider embedding $S$ into $T$ to form the surface $S'$, so that in tetrahedron $\tau$ we have that $S'_\tau$ is isotopic to $S_G$. This is not necessarily a normal surface, so we attempt to normalise it. As $S'$ is a sphere, any tubed piece within $\tau$ may be cut and `capped' (see Section \ref{appNORMAL}), cutting $S'$ into two spheres. In particular, we find an inner-most tube in $S'_\tau$, and cut and cap it as in Figure \ref{fig:cuttingCapping}.

    \begin{figure}[H]
        \centering
        \includegraphics[width=0.7\textwidth]{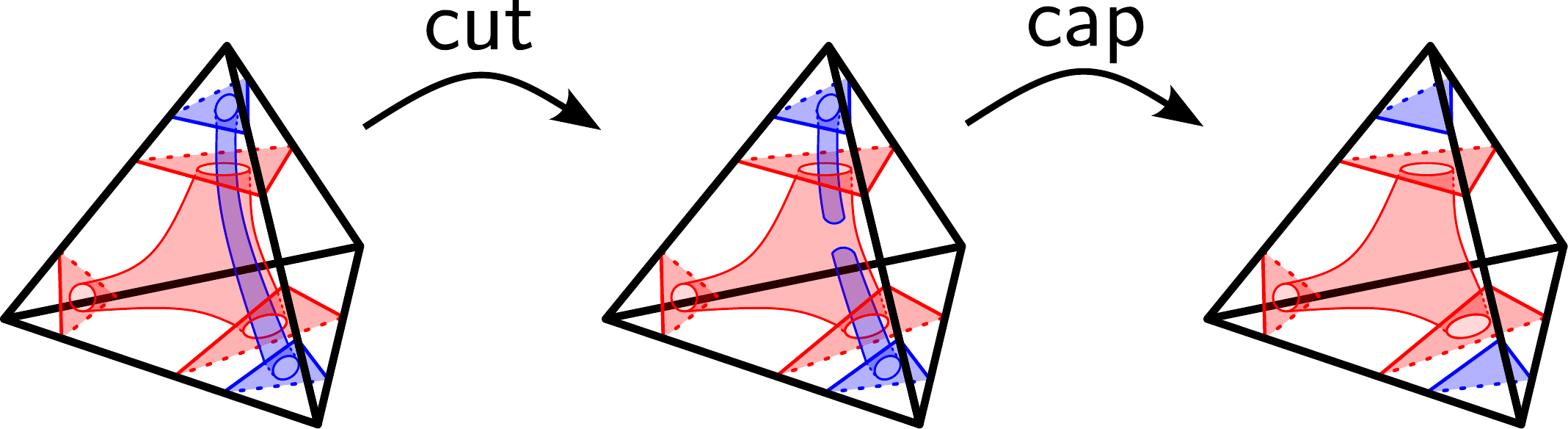}
        \caption{Cutting and then capping an inner-most tube}
        \label{fig:cuttingCapping}
    \end{figure}

    As we repeat this process, we form $\bar{S'}$, a collection of spheres, with $\partial \bar{S'}_\tau =\partial S_G$. The components of $\bar{S'}_\tau$ are triangles, plus possibly either octagons or quadrilaterals (but not both as they are not locally compatible).

    If $\bar{S'}$ has one or more octagon piece, then we normalise all but one octagon by pushing them towards adjacent tetrahedra (see Figure \ref{fig:octagonCompression} in Section \ref{appNORMAL}). Now, since exactly one octagon remains, $\bar{S'}$ has a component that is an almost normal sphere in $T$. From Proposition 5.12 of \cite{MR2057531}, this means $T$ must be $\mathbb{S}^3$ -- but this contradicts our assumptions about $T$.

    Now, if $\bar{S'}$ has quadrilaterals but no octagons, then it has a component that is a normal 2-sphere in $T$ which is not vertex linking. This contradicts $T$ being zero-efficient.

    Finally, consider the case where $\bar{S'}$ has only triangular pieces. As we assumed that $S$ is not vertex linking, then $S_G$ must have had at least one copy of an annulus between triangles. But, as $T'$ is a one-vertex triangulation, this means that $S$ had genus $>1$, as each tube essentially adds a handle to a vertex linking sphere formed from the remaining triangles. But, then $S$ is not a sphere, so we again find a contradiction.

    In each case, we have reached a contradiction to the assumption that $T'$ is not zero-efficient. Hence, we can conclude that the gadget preserves zero-efficiency.

\end{proof}
Now, we address the parallel annulus types that are not captured by the gadget, and show why this exceptional case is harmless for algorithms such as computing Heegaard genus.
\begin{observation}
    Suppose $S$ is an almost normal surface in some closed triangulation $T$. If the almost normal piece in $S$ is an annulus in tetrahedron $\tau$, then on any face of $\tau$ that intersects the annulus twice, the annulus may be isotoped (`pushed across') to the adjacent tetrahedron.
\end{observation}
For example, consider such a case in Figure \ref{fig:pushOff2}.

\begin{figure}[H]
    \centering
    \includegraphics[width=0.7\textwidth]{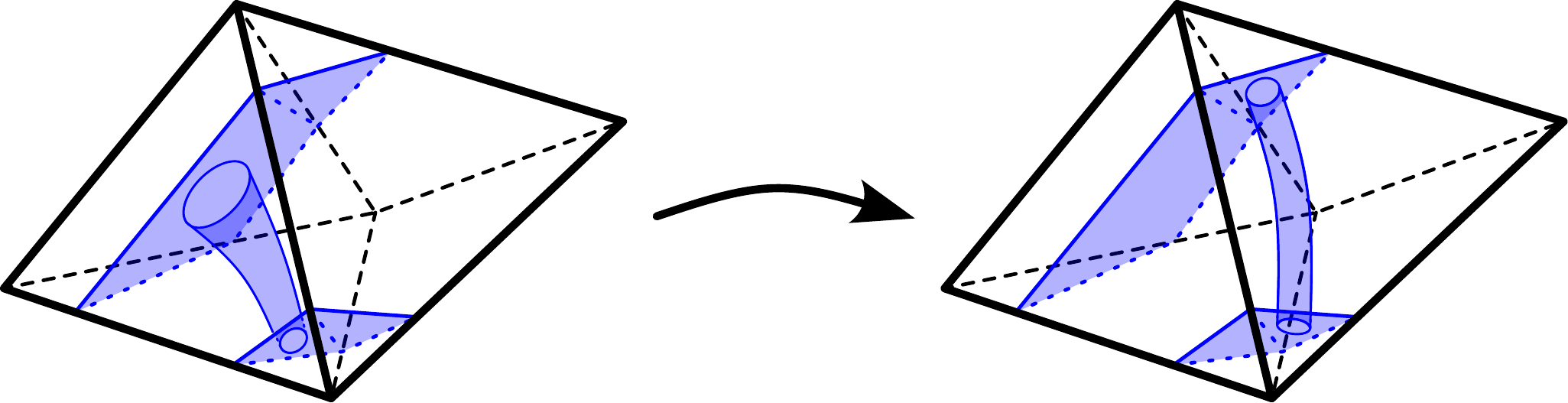}
    \caption{An annulus piece between a triangle and quadrilateral is pushed into a neighbouring tetrahedron, as an annulus between two non-parallel triangles.}
    \label{fig:pushOff2}
\end{figure}
The normal discs in the adjacent tetrahedron need not be the same type, as in Figure \ref{fig:pushOff2}.

Whether pairs of boundary curves constrained to the shared face are parallel determines what annulus piece may appear in this adjacent tetrahedron. Which face the tube is pushed towards determines which annulus may appear in the adjacent tetrahedron.

\begin{lemma}\label{parallelLemma}
    Let $S$ be a connected, almost normal surface in a closed, orientable triangulation $T$ (which is not the 3-sphere). Suppose it has an annulus piece of parallel triangle or quadrilateral type in some tetrahedron $\tau$. Then, one of the following holds: 
    \begin{enumerate}
        \item $S$ is isotopic to an almost normal surface $S'$ with a non-parallel boundary type annulus in some (possibly different) tetrahedron $\tau'$; \label{enum1}
        \item $S$ does not represent a minimal Heegaard splitting; \label{enum2}
        \item $S$ is the orientable double cover of some non-orientable surface $U_{no}$ with a tube connecting the two sheets that cover some local region of $U_{no}$. \label{enum3}
    \end{enumerate}
\end{lemma}
\begin{proof}

    Suppose we have a surface $S$ with an annulus of either parallel-triangle or parallel-quadrilateral type. We attempt to push the tube from tetrahedron to tetrahedron until reaching a non-parallel annulus type. If successful, this proves Case \ref{enum1}.

    If a non-parallel piece cannot be found, then our surface must either be two parallel copies of an orientable surface $U_{o}$, or the orientable double cover of a non-orientable surface $U_{no}$, in both cases with a tube connecting its two sheets. The non-orientable case is simply Case \ref{enum3}, which will be discussed after the proof.
    
    Suppose that $S$ realises a Heegaard splitting of genus $g$ and therefore $S$ has Euler characteristic $\chi(S)=2-2g$.

    If $S=U_o\# U_o$ where $U_o$ is orientable and separating, then $U_o$ has Euler characteristic $\chi(U_o)=2-g$ and $U_o$ splits the triangulation into 3-manifolds $X$ and $Y$ where $\partial X=\partial Y=U_o$. Refer to Figure \ref{fig:separatedRegions}. Now, consider the 3-manifold $Z$ bounded by $S=U_o\# U_o$, wedged between the two copies of $U_o$. As $\chi(U_o\# U_o)=\chi(S)=2-2g$ for the Heegaard splitting $S$, $Z$ has genus $g$. To the other side of $S$ is the 3-manifold $X\asymp Y$, which is $X$ connected to $Y$ by a solid tube. As $S$ is a Heegaard splitting, $X\asymp Y$ must also be a handlebody of genus $g$. But, then $X$ and $Y$ are both handlebodies of genus $\frac{g}{2}$ (because they are both bounded individually by the genus $g$ surface $U_o$), so $U_o$ is a Heegaard splitting of genus $\frac{g}{2}$. Hence, $S$ is not a minimal Heegaard splitting (unless $g=0$, but $T$ is not the 3-sphere).

    \begin{figure}[H]
        \centering
        \includegraphics[width=\textwidth]{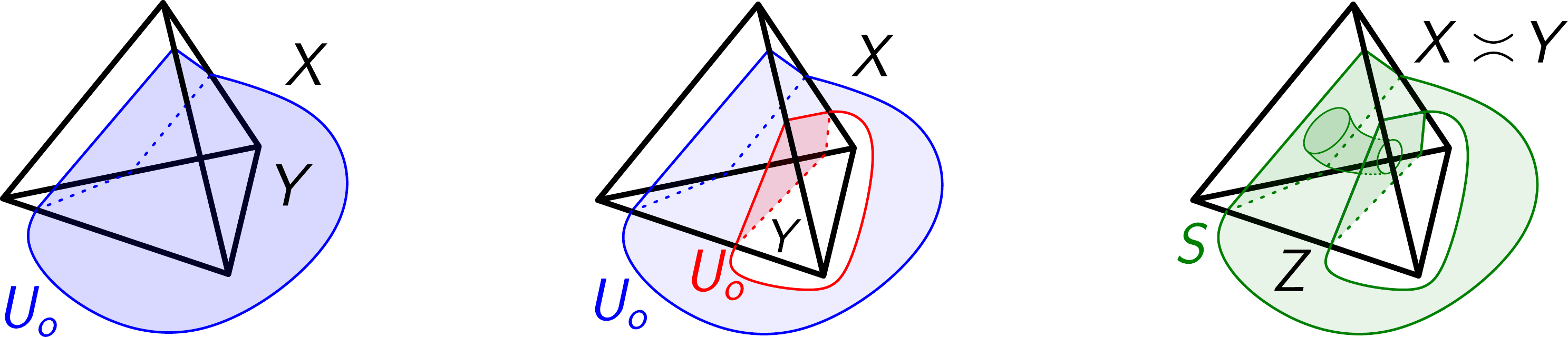}
        \caption{An orientable surface $U_o$ divides the 3-manifold into $X$ and $Y$. Two copies of $U_o$, tubed together, form $S=U_o\# U_o$ which divide the 3-manifold into $Z$ and $X\asymp Y$.}
        \label{fig:separatedRegions}
    \end{figure}

    Finally, if $S=U_o\#U_o$ but $U_o$ is non-separating, then cutting $T$ along $S$ yields a 3-manifold $X$ with a properly embedded disc (slicing through the tube) that separates $\partial X$ but not $X$. Therefore $X$ is not a handlebody and $S$ is not a Heegaard splitting, giving us Case~2.
 
\end{proof}
    Consider Case \ref{enum3} of the previous lemma. We abuse notation and let $S=U_{no}\asymp U_{no}$ represent the orientable double cover of $U_{no}$ with a tube connecting its two sheets. Cutting the triangulation $T$ along $U_{no}$ yields the 3-manifold $X$ with $\partial X=S$. Since $S$ cuts $T$ into two handlebodies of genus $g$, it follows that $X$ is just one of these handlebodies with one of its handles cut. That is, $X$ is a handlebody with genus $g-1$. Such a case can be detected by searching for surfaces with Euler characteristic $\chi(U_{no})=\frac{1}{2}\chi(S)+1$ which cut the triangulation into one genus $g-1$ handlebody.

\begin{corollary}\label{cor37}

    For an almost normal surface $S$ that realises a minimal Heegaard splitting of a triangulation $T$, one of the following holds:
    \begin{enumerate}
        \item There exists a gluing of the gadget into $T$ to form $T'=T_\tau^{abcd(G)}$ for some permutation $(a,b,c,d)$ and some tetrahedron $\tau$, and there exists a normal surface $S'$ in $T'$ that is isotopic to $S$;
        \item $S$ is of the type of Case \ref{enum3} in Lemma \ref{parallelLemma}.
    \end{enumerate}
    
\end{corollary}
\begin{proof}
    
    Suppose that $S$ is an almost normal surface in a triangulation $T$, with its almost normal piece in tetrahedron $\tau$.

    There are twelve rotational symmetries of a tetrahedron (twelve even permutations of its vertices), and hence twelve possible triangulations that may be formed for each tetrahedron.

    We may transform the coordinates of the normal pieces as in Corollary \ref{normalCoords}.

    If the almost normal piece of $S$ is an annulus between parallel triangles or quadrilaterals, then we attempt to push this tube across the tetrahedra until reaching a non-parallel annulus piece, and apply the gadget accordingly. If this is unsuccessful, then $S$ must be the orientable double cover of some non-orientable normal surface $U_{no}$ with a tube connecting its two sheets, as in Case \ref{enum3} of Lemma \ref{parallelLemma}. Then, to determine if $S$ is a genus $g$ Heegaard splitting, simply cut along $U_{no}$ and check, via the algorithm for handlebody recognition \cite{MR4604006}, whether the resulting manifold is a genus $g-1$ handlebody.

    Now, if the almost normal piece of $S$ is an octagon or a non-parallel annulus piece, then we seek to construct a surface in $T_\tau^{abcd(G)}$ for some permutation of the vertices of $\tau$ into $G$, $(0,1,2,3)\mapsto(a,b,c,d)$. 

    As the gadget has two distinct triangle-quadrilateral annulus pieces of a total possible twelve, each permutation of the gadget can represent two distinct annuli of this type. Hence, we only need to try six possible permutations. We choose the following,

    \[(0,1,2,3),\quad(0,3,1,2),\quad (1,0,3,2),\quad(1,2,0,3),\quad(2,0,1,3),\quad(3,1,0,2).\]
    This choice of permutation allows us to `rotate' the required octagon or annulus piece of $G$ into the same position as the one in $\tau$. Refer to Appendix \ref{appCONSTRUCTION} for the specific choices of individual almost normal pieces.

\end{proof}

\subsection{Implementations}

In practice, as there is no current implementation of annulus pieces in \textit{Regina}, we cannot directly normalise an almost normal surface using the gadget. Instead, we need to consider all \textit{normal} surfaces from all possible triangulations resulting from gluing the gadget into any tetrahedron.

With the gadget, Heegaard splittings can now always be represented by normal surfaces. However, these surfaces are not necessarily fundamental. Therefore we use the gadget to generate a set of candidate surfaces for each potential value of genus; this set will be finite, due to Euler characteristic being additive under Haken sum. We discuss this process further in Section \ref{heeeeSec}.

\subsubsection{Method 1: Generating All Surfaces}\label{meth1Head}

For each tetrahedron $\tau$ in a triangulation $T$, form $T_\tau ^{abcd(G)}$ for the six distinct permutations (see the proof of Corollary \ref{cor37}), and then generate all fundamental normal surfaces in it. For a triangulation with $n$ tetrahedra, this means $6n$ sets of normal surfaces must be enumerated. This represents almost normal surfaces that may be embedded in $T$, as well as many others (which use surfaces of the gadget that are not tubes, octagons, triangles or quadrilaterals).

Denote the running time for generating normal surfaces for a triangulation with $n$ tetrahedra by $T(n)$, and let $C(n)$ represent the number of normal surfaces for such a triangulation. From Section \ref{runningtimeseccie}, $T(n)\in\mathcal{O}(1.934^n)$ and $C(n)\in\mathcal{O}(1.503^n)$. Suppose that the running time of the algorithm to check if a surface is a Heegaard splitting has running time $f(n)$. Then, to generate and test each surface in the $6n$ triangulations, we find a running time on the order of $M_1(n)=6n(T(n+4)+C(n+4)f(n))=6n(1.934^{n+4}+1.503^{n+4}f(n))$.

\subsubsection{Method 2: Constructing Each Annulus}\label{meth2Head}
Alternatively, we have developed a method to form the necessary set of almost normal surfaces using an existing set of fundamental normal surfaces for the base triangulation.

Let $\Omega$ be the set of fundamental normal surfaces for a triangulation $T$. For each surface $S\in\Omega$, we form a set of almost normal surfaces with an annulus piece between triangles and/or quadrilaterals in $S$ as follows. For each tetrahedron $\tau$, if $S_\tau$ has at least two triangles on different vertices, then we consider the surface with an annulus between two of these triangles. Similarly, if $S_\tau$ has at least one triangle and at least one quadrilateral, we consider the surface with an annulus between these. In either case, we determine which annulus type from the gadget is necessary, and under which permutation, and construct the new surface within $T'=T_\tau^{abcd(G)}$. As we begin with a set of normal surfaces (that is, they are not already `tubed'), when we form the equivalent surface in $T'$, we must add the coordinates for the annulus piece to the coordinates for each triangle or quadrilateral excluding those corresponding to the boundaries of the annulus. 

For example, in Figure \ref{fig:howToTubeTemp}, $S_\tau$ has coordinates $(1,0,0,2,0,0,1)$ and supports several annulus types, such as one connecting a triangle to its quadrilateral. Using the identity permutation, we find that $G_{t_3}^{q_{03/12}}$ gives us the annulus between \texttt{tri\_3:quad\_03/12}. Then, the remaining triangle pieces are represented using $G_{t_0}$ and $G_{t_3}$.
\begin{figure}[H]
    \centering
    \includegraphics[width=\textwidth]{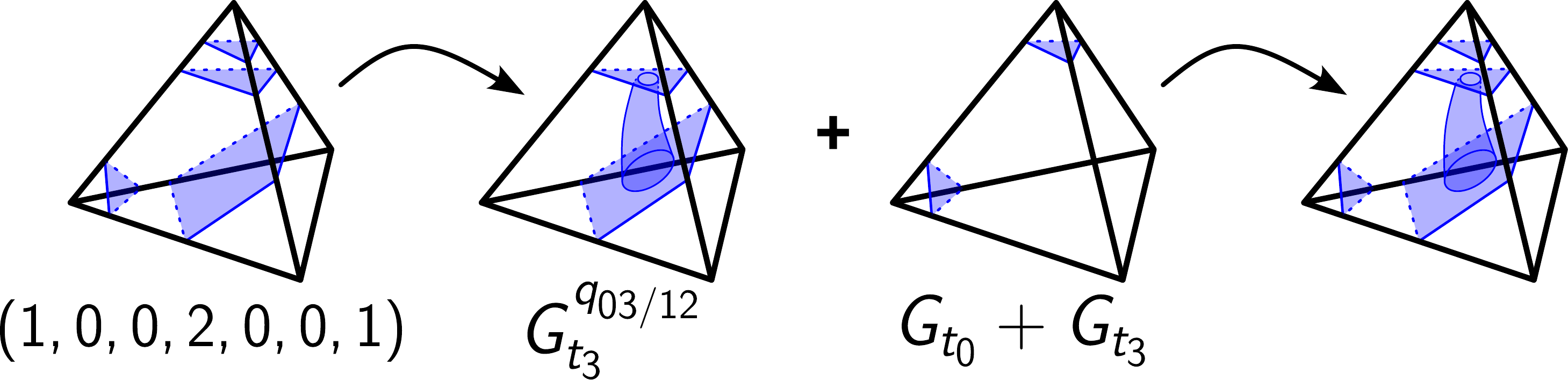}
    \caption{A surface with coordinates $(1,0,0,2,0,0,1)$ in a tetrahedron, and its `tubing'.}
    \label{fig:howToTubeTemp}
\end{figure}
A complete guide to our choice of permutations and coordinates is given in Appendix \ref{appCONSTRUCTION}.

Now, using the same results from Section \ref{runningtimeseccie}, the running time for this method is on the order of $M_2(n)=T(n)+C(n)*6n*f(n)=1.934^n+6n\times1.503^n f(n)$, as there are at most six annulus types that could exist in any given tetrahedron.

\paragraph*{Performance Comparison}
In our discussion of the two methods, we found $M_1(n)\in\mathcal{O}(6n(1.934^{n+4}+1.503^{n+4}f(n)))$ and $M_2(n)\in\mathcal{O}(1.934^n+6n\times1.503^n f(n))$, where $f(n)$ is the time to test if a given surface is a Heegaard splitting.
Asymptotically $M_2(n)\sim 6n\times1.934^4\times M_1(n)$, and so we expect method 2 to be significantly faster than method 1.

\paragraph*{Generating \textit{All} Surfaces with Method 2}
In Method 2 (Section \ref{meth2Head}), we described how a non-parallel annulus of an almost normal surface can be converted into a normal surface using the gadget. Here, we justify that this method produces all the candidate almost normal surfaces that may represent a minimal Heegaard splitting. First, suppose that almost normal surfaces with octagon pieces have been considered already (i.e. have been tested as Heegaard splittings). Now, any almost normal surface with an annulus piece has such an annulus `tubing' together two triangles and/or quadrilaterals. 

Suppose $S$ is an orientable fundamental almost normal surface with Euler characteristic $\chi(S)<0$, with a non-parallel annulus in some tetrahedron $\tau$. Let $\bar{S}$ be the surface generated by cutting and capping the annulus. Of course, $\chi(\bar{S})=\chi(S)+2$ and $\bar{S}$ is a normal surface, but is not necessarily fundamental, so $\bar{S}=\sum_{i} A_i$ for some fundamental normal surfaces $A_i$. The annulus of $S$ either connects two components in $\tau$ of a particular $A_i$, or connects two components between $A_j$ and $A_k$ for $j\neq k$. We can generate all tubings of $A_i$ to itself by constructing tubes as in Section \ref{meth2Head}. Similarly, we can generate tubings of $A_j+A_k$.

Now, if $S$ has a parallel annulus piece, then either it is of the type of Case \ref{enum3} in Lemma \ref{parallelLemma}, or $S$ will not represent a minimal Heegaard splitting, or it is isotopic to some $S'$ with a non-parallel annulus.

In practice, this results in the following method for detecting a splitting of genus $g$. 
\begin{enumerate}
    \item Generate the set $\Omega_{AN}$ of all fundamental almost normal surfaces of a triangulation $T$ in \textit{Regina}, test if those with genus $g$ are Heegaard splittings. If not, generate the set of fundamental normal surfaces, $\Omega_N$.
    \item For each tetrahedron, for each orientable surface of genus $g-1$ in $\Omega_N$, form all tubings.
    \item For each tetrahedron, for each set of locally compatible surfaces in $\Omega_N$ whose Euler characteristics sum to $2-2(g-1)$, form all tubings.
    \item Test all tubed surfaces as Heegaard splittings.
\end{enumerate}
Note that constructing a tube increases the Euler characteristic of a surface by 2, which is equivalent to increasing the genus of a connected surface by 1.

\section{Computing Heegaard Genus}\label{heeeeSec}
As discussed in Section \ref{heegen}, a result of Rubinstein declares that Heegaard splittings are given by $S=S_{an}+S_t$ where $S_{an}$ is almost normal, $S_t$ is normal, $\chi(S_{an})<0$ and $\chi(S_t)=0$ \cite{MR1470718}. So, to determine if a triangulation has a genus $g$ Heegaard splitting, all such surfaces where $\chi(S_{an})=2-2g$ must be tested as splittings. Furthermore, we know that every summand of $S_{an}$ must have negative Euler characteristic, and every summand of $S_t$ has zero Euler characteristic. Hence, there are finitely many cases to consider of $S_{an}$ as their Euler characteristic must represent an integer partition of $2-2g$. We can generate these using normal and almost normal surfaces in \textit{Regina} with octagon pieces, or by constructing tubed surfaces as in Section \ref{meth2Head}.

For the torus piece $S_t$, as Euler characteristic is additive under Haken sum, this could theoretically require an unbounded number of different sums of tori. However, from Lemma 5.1 of \cite{MR4252195}, we know that if $S$ is a normal surface and $A$ is an edge-linking torus, then $S+A$ is either disconnected or is isotopic to $S$, so need not be considered in summands of $S_t$.
Tentatively and naively, we therefore consider splittings as $S=S_{an}$ only.
If there are non-edge-linking tori present, this means we may not see our splitting, and so in such settings we can only form an upper bound on the genus.
Experimentally, edge-linking tori do form the majority of tori in the cases we have tested, and in those cases that remain, we can attempt to form a lower bound using algebraic techniques as in Section \ref{heegen}.

We have tested this approach on all orientable triangulations in the closed hyperbolic census of Hodgson and Weeks (in \cite{MR1341719}) -- all of which are zero-efficient. We first search for a genus 2 splitting, and if one does not exist, search for a genus 3 splitting (then 4, etc). Of the 3,000 triangulations in this census, 44 had $\mathrm{rank}(H_1(M))=3$, and genus 3 splittings were found. Then, 2,661 had $\mathrm{rank}(H_1(M))\leq2$, and genus 2 splittings were found. For each of the remaining 295, an exhaustive search of genus 2 surfaces (without tori) found no splittings, but genus 3 splittings were found. In these 295 cases, we cannot guarantee they have genus 3, but merely provide an informed upper bound.

\bibliography{the_extra_b_is_for}

\appendix

\section{Running Time Computations}\label{appTIMES}
Of the first 3000 triangulations in the \textit{Hodgson-Weeks Closed Hyperbolic Census} in \textit{Regina}, fundamental normal surfaces were generated for the 2974 triangulations with at most 23 tetrahedra. The number of surfaces and enumeration time were recorded. These tests were run on an 8-core CPU with 16GB RAM.

Using \texttt{curve\_fit} in \textit{SciPy} \cite{2020SciPy-NMeth}, we fit linear models to $(n,\log_{10}(T))$ and $(n,\log_{10}(S)$ where $n$ is the number of tetrahedra, $T$ is the running time and $S$ is the number of normal surfaces. We found $\log_{10}(T)=0.2866n-4.2309$ with $R^2=0.8363$ for the enumeration time, and $\log_{10}(S)=0.1771n-0.3887$ with $R^2=0.8384$ for the number of surfaces. Equivalently, $T\approx5.9\times10^{-5} 1.934^n$ and $S\approx 0.41\times 1.503^n$. By this, we may estimate that $T(n)\in\mathcal{O}(1.934^n)$ and $C(n)\in\mathcal{O}(1.503^n)$.

\begin{figure}[H]
    \centering
    \includegraphics[width=0.45\textwidth]{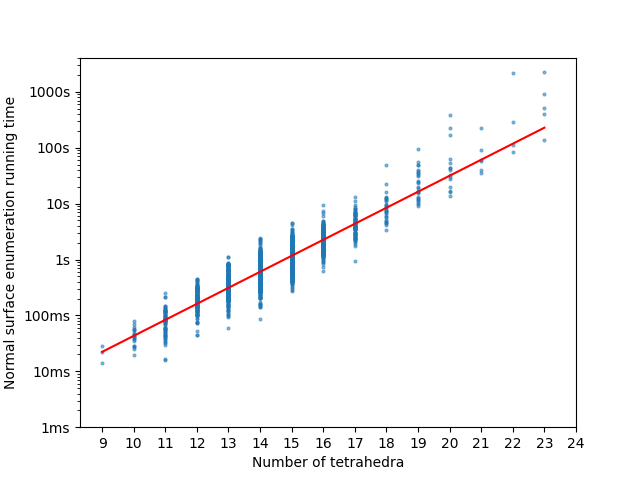} \includegraphics[width=0.45\textwidth]{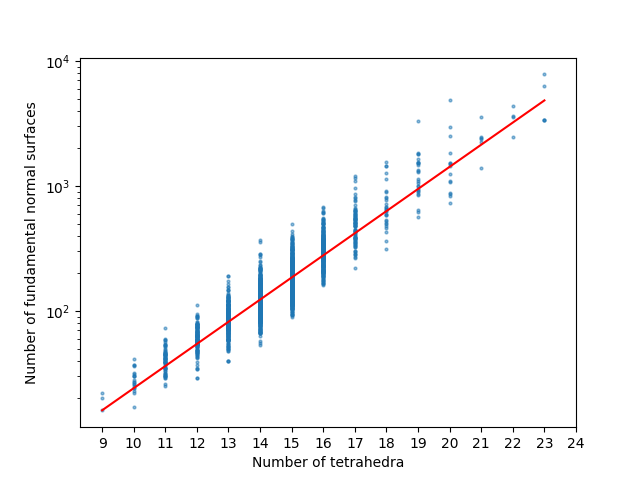}
    \caption{Data plotted with the exponential model for the enumeration time (left) and number of surfaces (right).}
    \label{fig:dataGraphs}
\end{figure}

\section{Fundamental Normal Surfaces of the Gadget}\label{appSURFTABLE}
The coordinates of the surfaces introduced in Table \ref{tab:surfDetails} are detailed below. The coordinates and order of the surfaces are as produced in version 7.3 of \textit{Regina}, for the triangulation with isomorphism signature \texttt{`fHLMabddeaaaa'}. To replicate these surfaces in \textit{Regina}, use the command \\ \texttt{NormalSurfaces(Triangulation3(`fHLMabddeaaaa'),\,NS\_STANDARD,\,NS\_FUNDAMENTAL)}.
\begin{table}[H]
\begin{tabular}{|l|l|}
\hline
Surface                          & Coordinates                                                                                       \\ \hline
\texttt{sf\_0}  & $(0,0,0,0,0,0,1,\,\,\,0,0,0,0,0,0,1,\,\,\,0,0,0,0,0,0,1,\,\,\,0,0,0,0,0,0,1,\,\,\,0,0,0,0,0,0,1)$ \\ \hline
\texttt{sf\_1}  & $(0,0,0,0,0,0,1,\,\,\,0,0,0,0,0,0,1,\,\,\,0,0,0,0,0,0,1,\,\,\,0,0,0,0,0,0,1,\,\,\,0,1,1,0,0,0,0)$ \\ \hline
\texttt{sf\_2}  & $(0,0,0,0,0,0,1,\,\,\,0,0,0,0,0,0,1,\,\,\,1,0,0,1,0,0,0,\,\,\,0,0,0,0,0,0,1,\,\,\,1,0,0,1,0,0,0)$ \\ \hline
\texttt{sf\_3}  & $(0,0,0,0,0,1,0,\,\,\,0,0,0,0,0,1,0,\,\,\,0,0,0,0,0,1,0,\,\,\,0,0,0,0,0,1,0,\,\,\,0,0,0,0,0,1,0)$ \\ \hline
\texttt{sf\_4}  & $(0,0,0,0,0,1,0,\,\,\,0,0,0,0,0,1,0,\,\,\,0,1,0,1,0,0,0,\,\,\,0,1,0,1,0,0,0,\,\,\,0,1,0,1,0,0,0)$ \\ \hline
\texttt{sf\_5}  & $(0,0,0,0,0,1,0,\,\,\,1,0,1,0,0,0,0,\,\,\,0,0,0,0,0,1,0,\,\,\,1,0,1,0,0,0,0,\,\,\,0,0,0,0,0,1,0)$ \\ \hline
\texttt{sf\_6}  & $(0,0,0,0,1,0,0,\,\,\,0,0,0,0,1,0,0,\,\,\,0,0,0,0,1,0,0,\,\,\,0,0,0,0,1,0,0,\,\,\,0,0,0,0,1,0,0)$ \\ \hline
\texttt{sf\_7}  & $(0,0,0,1,0,0,0,\,\,\,0,0,0,1,0,0,0,\,\,\,0,0,0,1,0,0,0,\,\,\,0,0,0,1,0,0,0,\,\,\,0,0,0,1,0,0,0)$ \\ \hline
\texttt{sf\_8}  & $(0,0,1,0,0,0,0,\,\,\,0,0,1,0,0,0,0,\,\,\,0,0,1,0,0,0,0,\,\,\,0,0,1,0,0,0,0,\,\,\,0,0,1,0,0,0,0)$ \\ \hline
\texttt{sf\_9}  & $(0,0,1,0,0,1,0,\,\,\,0,0,1,0,0,1,0,\,\,\,0,1,1,1,0,0,0,\,\,\,0,1,1,1,0,0,0,\,\,\,0,0,0,1,0,0,1)$ \\ \hline
\texttt{sf\_10} & $(0,1,0,0,0,0,0,\,\,\,0,1,0,0,0,0,0,\,\,\,0,1,0,0,0,0,0,\,\,\,0,1,0,0,0,0,0,\,\,\,0,1,0,0,0,0,0)$ \\ \hline
\texttt{sf\_11} & $(0,1,0,1,0,0,0,\,\,\,0,1,0,1,0,0,0,\,\,\,0,0,0,0,0,1,0,\,\,\,0,0,0,0,0,1,0,\,\,\,0,0,0,0,0,1,0)$ \\ \hline
\texttt{sf\_12} & $(0,1,1,0,0,0,0,\,\,\,0,1,1,0,0,0,0,\,\,\,0,1,1,0,0,0,0,\,\,\,0,1,1,0,0,0,0,\,\,\,0,0,0,0,0,0,1)$ \\ \hline
\texttt{sf\_13} & $(1,0,0,0,0,0,0,\,\,\,1,0,0,0,0,0,0,\,\,\,1,0,0,0,0,0,0,\,\,\,1,0,0,0,0,0,0,\,\,\,1,0,0,0,0,0,0)$ \\ \hline
\texttt{sf\_14} & $(1,0,0,0,0,1,0,\,\,\,1,0,0,0,0,1,0,\,\,\,0,1,0,0,0,0,1,\,\,\,1,1,0,1,0,0,0,\,\,\,0,1,0,0,0,0,1)$ \\ \hline
\texttt{sf\_15} & $(1,0,0,0,0,1,0,\,\,\,1,0,0,0,0,1,0,\,\,\,0,1,0,0,0,0,1,\,\,\,1,1,0,1,0,0,0,\,\,\,0,2,1,0,0,0,0)$ \\ \hline
\texttt{sf\_16} & $(1,0,0,1,0,0,0,\,\,\,1,0,0,1,0,0,0,\,\,\,0,0,0,0,0,0,1,\,\,\,1,0,0,1,0,0,0,\,\,\,0,0,0,0,0,0,1)$ \\ \hline
\texttt{sf\_17} & $(1,0,0,1,0,0,0,\,\,\,1,0,0,1,0,0,0,\,\,\,0,0,0,0,0,0,1,\,\,\,1,0,0,1,0,0,0,\,\,\,0,1,1,0,0,0,0)$ \\ \hline
\texttt{sf\_18} & $(1,0,1,0,0,0,0,\,\,\,0,0,0,0,0,1,0,\,\,\,1,0,0,1,1,0,0,\,\,\,0,1,0,1,0,0,0,\,\,\,1,0,0,1,1,0,0)$ \\ \hline
\texttt{sf\_19} & $(1,0,1,0,0,0,0,\,\,\,0,0,0,0,0,1,0,\,\,\,1,0,1,0,0,0,0,\,\,\,0,0,0,0,0,1,0,\,\,\,1,0,1,0,0,0,0)$ \\ \hline
\texttt{sf\_20} & $(1,0,1,0,0,1,0,\,\,\,0,0,0,0,0,2,0,\,\,\,0,1,1,0,0,0,1,\,\,\,0,1,0,1,0,1,0,\,\,\,0,0,0,0,0,0,2)$ \\ \hline
\texttt{sf\_21} & $(1,0,1,0,0,1,0,\,\,\,0,0,0,0,0,2,0,\,\,\,0,1,1,0,0,0,1,\,\,\,0,1,0,1,0,1,0,\,\,\,0,1,1,0,0,0,1)$ \\ \hline
\texttt{sf\_22} & $(1,0,1,0,0,1,0,\,\,\,0,0,0,0,0,2,0,\,\,\,0,1,1,0,0,0,1,\,\,\,0,1,0,1,0,1,0,\,\,\,0,2,2,0,0,0,0)$ \\ \hline
\texttt{sf\_23} & $(1,0,1,0,0,1,0,\,\,\,0,0,0,0,0,2,0,\,\,\,1,1,1,1,0,0,0,\,\,\,0,1,0,1,0,1,0,\,\,\,1,0,0,1,0,0,1)$ \\ \hline
\texttt{sf\_24} & $(1,0,1,0,0,1,0,\,\,\,1,0,1,0,0,1,0,\,\,\,0,1,1,0,0,0,1,\,\,\,1,1,1,1,0,0,0,\,\,\,0,0,0,0,0,0,2)$ \\ \hline
\texttt{sf\_25} & $(1,0,1,1,0,0,0,\,\,\,0,0,0,1,0,1,0,\,\,\,0,0,1,0,0,0,1,\,\,\,0,0,0,1,0,1,0,\,\,\,0,0,1,0,0,0,1)$ \\ \hline
\texttt{sf\_26} & $(1,0,1,1,0,0,0,\,\,\,0,0,0,1,0,1,0,\,\,\,0,0,1,0,0,0,1,\,\,\,0,0,0,1,0,1,0,\,\,\,0,1,2,0,0,0,0)$ \\ \hline
\texttt{sf\_27} & $(1,1,1,0,0,0,0,\,\,\,0,1,0,0,0,1,0,\,\,\,1,1,1,0,0,0,0,\,\,\,0,1,0,0,0,1,0,\,\,\,1,0,0,0,0,0,1)$ \\ \hline
\texttt{sf\_28} & $(1,1,1,1,0,0,0,\,\,\,0,1,0,1,0,1,0,\,\,\,0,1,1,0,0,0,1,\,\,\,0,1,0,1,0,1,0,\,\,\,0,0,0,0,0,0,2)$ \\ \hline
\end{tabular}
\caption{The coordinates of the fundamental normal surfaces of the gadget.}
\label{tab:surfCoords}
\end{table}

\section{Constructing Normal Surfaces in the Gadget}\label{appCONSTRUCTION}
Given a normal surface $S$ in a triangulation $T$, we describe how to build a surface $S'$ in $T'=T_\tau^{abcd(G)}$ where outside of a tetrahedron $\tau$, $\mathbf{S_{T\setminus\tau}}=\mathbf{S'_{T'\setminus G}}$, and where $S'_G$ is isotopic to $S_\tau$ with a tube connecting two non-parallel pieces. In practice, normal surfaces of $T$ are generated, and then all valid tubes are constructed (i.e. non-zero components and local compatibility are satisfied). Table \ref{tab:surfCoordsTubings} lists the required permutations and the new coordinates for $S'_G$ in each case.

\begin{table}[H]
\begin{tabular}{|l|l|l|l|}\hline
Tube in $\tau$                            & Tube in $G$                               & Perm. & New coordinates \\\hline
\texttt{tri\_0:tri\_1} & $G_{t_0}^{t_2}$ & $(2,0,1,3)$ & $\begin{aligned}&(t_1-1)\mathbf{G_{t_0}} + t_2\mathbf{G_{t_1}} + (t_0-1)\mathbf{G_{t_2}} + t_3\mathbf{G_{t_3}} \\ &+ q_{01/23}\mathbf{G_{q_{02/13}}} + \mathbf{G_{t_0}^{t_2}}\end{aligned}$                 \\\hline

\texttt{tri\_0:tri\_2} & $G_{t_0}^{t_2}$ & $(0,1,2,3)$ & $\begin{aligned}&(t_0-1)\mathbf{G_{t_0}} + t_1\mathbf{G_{t_1}} + (t_2-1)\mathbf{G_{t_2}} + t_3\mathbf{G_{t_3}} \\ &+ q_{02/13}\mathbf{G_{q_{02/13}}} + \mathbf{G_{t_0}^{t_2}}\end{aligned}$               \\\hline

\texttt{tri\_0:tri\_3} & $G_{t_0}^{t_3}$ & $(0,1,2,3)$ & $\begin{aligned}&(t_0-1)\mathbf{G_{t_0}} + t_1\mathbf{G_{t_1}} + t_2\mathbf{G_{t_2}} + (t_3-1)\mathbf{G_{t_3}} \\ &+ q_{03/12}\mathbf{G_{q_{03/12}}} + \mathbf{G_{t_0}^{t_3}}\end{aligned}$                \\\hline

\texttt{tri\_1:tri\_2} & $G_{t_1}^{t_2}$ & $(0,1,2,3)$ & $\begin{aligned}&t_0\mathbf{G_{t_0}} + (t_1-1)\mathbf{G_{t_1}} + (t_2-1)\mathbf{G_{t_2}} + t_3\mathbf{G_{t_3}} \\ &+ q_{03/12}\mathbf{G_{q_{03/12}}} + \mathbf{G_{t_1}^{t_2}}\end{aligned}$                \\\hline

\texttt{tri\_1:tri\_3} & $G_{t_1}^{t_3}$ & $(0,1,2,3)$ & $\begin{aligned}&t_0\mathbf{G_{t_0}} + (t_1-1)\mathbf{G_{t_1}} + t_2\mathbf{G_{t_2}} + (t_3-1)\mathbf{G_{t_3}} \\ &+ q_{02/13}\mathbf{G_{q_{02/13}}} + \mathbf{G_{t_1}^{t_3}}\end{aligned}$                \\\hline

\texttt{tri\_2:tri\_3} & $G_{t_1}^{t_3}$ & $(2,0,1,3)$ & $\begin{aligned}&t_1\mathbf{G_{t_0}} + (t_2-1)\mathbf{G_{t_1}} + t_0\mathbf{G_{t_2}} + (t_3-1)\mathbf{G_{t_3}} \\ &+ q_{01/23}\mathbf{G_{q_{02/13}}} + \mathbf{G_{t_1}^{t_3}}\end{aligned}$                \\\hline

\texttt{tri\_0:quad\_01/23} & $G_{t_0}^{q_{03/12}}$ & $(0,3,1,2)$ & $\begin{aligned}&(t_0-1)\mathbf{G_{t_0}} + t_2\mathbf{G_{t_1}} + t_3\mathbf{G_{t_2}} + t_1\mathbf{G_{t_3}} \\ &+ (q_{01/23}-1)\mathbf{G_{q_{03/12}}} + \mathbf{G_{t_0}^{q_{03/12}}}\end{aligned}$                \\\hline

\texttt{tri\_0:quad\_02/13} & $G_{t_3}^{q_{03/12}}$ & $(3,1,0,2)$ & $\begin{aligned}&t_2\mathbf{G_{t_0}} + t_1\mathbf{G_{t_1}} + t_3\mathbf{G_{t_2}} + (t_0-1)\mathbf{G_{t_3}} \\ &+ (q_{02/13}-1)\mathbf{G_{q_{03/12}}} + \mathbf{G_{t_3}^{q_{03/12}}}\end{aligned}$                \\\hline

\texttt{tri\_0:quad\_03/12} & $G_{t_0}^{q_{03/12}}$ & $(0,1,2,3)$ & $\begin{aligned}&(t_0-1)\mathbf{G_{t_0}} + t_1\mathbf{G_{t_1}} + t_2\mathbf{G_{t_2}} + t_3\mathbf{G_{t_3}} \\ &+ (q_{03/12}-1)\mathbf{G_{q_{03/12}}} + \mathbf{G_{t_0}^{q_{03/12}}}\end{aligned}$                \\\hline

\texttt{tri\_1:quad\_01/23} & $G_{t_3}^{q_{03/12}}$ & $(0,3,1,2)$ & $\begin{aligned}&t_0\mathbf{G_{t_0}} + t_2\mathbf{G_{t_1}} + t_3\mathbf{G_{t_2}} + (t_1-1)\mathbf{G_{t_3}} \\ &+ (q_{01/23}-1)\mathbf{G_{q_{03/12}}} + \mathbf{G_{t_3}^{q_{03/12}}}\end{aligned}$                \\\hline

\texttt{tri\_1:quad\_02/13} & $G_{t_0}^{q_{03/12}}$ & $(2,0,1,3)$ & $\begin{aligned}&(t_1-1)\mathbf{G_{t_0}} + t_2\mathbf{G_{t_1}} + t_0\mathbf{G_{t_2}} + t_3\mathbf{G_{t_3}} \\ &+ (q_{02/13}-1)\mathbf{G_{q_{03/12}}} + \mathbf{G_{t_0}^{q_{03/12}}}\end{aligned}$                \\\hline

\texttt{tri\_1:quad\_03/12} & $G_{t_0}^{q_{03/12}}$ & $(1,0,3,2)$ & $\begin{aligned}&(t_1-1)\mathbf{G_{t_0}} + t_0\mathbf{G_{t_1}} + t_3\mathbf{G_{t_2}} + t_2\mathbf{G_{t_3}} \\ &+ (q_{03/12}-1)\mathbf{G_{q_{03/12}}} + \mathbf{G_{t_0}^{q_{03/12}}}\end{aligned}$                \\\hline

\texttt{tri\_2:quad\_01/23} & $G_{t_0}^{q_{03/12}}$ & $(1,2,0,3)$ & $\begin{aligned}&(t_2-1)\mathbf{G_{t_0}} + t_0\mathbf{G_{t_1}} + t_1\mathbf{G_{t_2}} + t_3\mathbf{G_{t_3}} \\ &+ (q_{01/23}-1)\mathbf{G_{q_{03/12}}} + \mathbf{G_{t_0}^{q_{03/12}}}\end{aligned}$                \\\hline

\texttt{tri\_2:quad\_02/13} & $G_{t_0}^{q_{03/12}}$ & $(3,1,0,2)$ & $\begin{aligned}&(t_2-1)\mathbf{G_{t_0}} + t_1\mathbf{G_{t_1}} + t_3\mathbf{G_{t_2}} + t_0\mathbf{G_{t_3}} \\ &+ (q_{02/13}-1)\mathbf{G_{q_{03/12}}} + \mathbf{G_{t_0}^{q_{03/12}}}\end{aligned}$                \\\hline

\texttt{tri\_2:quad\_03/12} & $G_{t_3}^{q_{03/12}}$ & $(1,0,3,2)$ & $\begin{aligned}&t_1\mathbf{G_{t_0}} + t_0\mathbf{G_{t_1}} + t_3\mathbf{G_{t_2}} + (t_2-1)\mathbf{G_{t_3}} \\ &+ (q_{03/12}-1)\mathbf{G_{q_{03/12}}} + \mathbf{G_{t_3}^{q_{03/12}}}\end{aligned}$                \\\hline

\texttt{tri\_3:quad\_01/23} & $G_{t_3}^{q_{03/12}}$ & $(1,2,0,3)$ & $\begin{aligned}&t_2\mathbf{G_{t_0}} + t_0\mathbf{G_{t_1}} + t_1\mathbf{G_{t_2}} + (t_3-1)\mathbf{G_{t_3}} \\ &+ (q_{01/23}-1)\mathbf{G_{q_{03/12}}} + \mathbf{G_{t_3}^{q_{03/12}}}\end{aligned}$                \\\hline

\texttt{tri\_3:quad\_02/13} & $G_{t_3}^{q_{03/12}}$ & $(2,0,1,3)$ & $\begin{aligned}&t_1\mathbf{G_{t_0}} + t_2\mathbf{G_{t_1}} + t_0\mathbf{G_{t_2}} + (t_3-1)\mathbf{G_{t_3}} \\ &+ (q_{02/13}-1)\mathbf{G_{q_{03/12}}} + \mathbf{G_{t_3}^{q_{03/12}}}\end{aligned}$                \\\hline

\texttt{tri\_3:quad\_03/12} & $G_{t_3}^{q_{03/12}}$ & $(0,1,2,3)$ & $\begin{aligned}&t_0\mathbf{G_{t_0}} + t_1\mathbf{G_{t_1}} + t_2\mathbf{G_{t_2}} + (t_3-1)\mathbf{G_{t_3}} \\ &+ (q_{03/12}-1)\mathbf{G_{q_{03/12}}} + \mathbf{G_{t_3}^{q_{03/12}}}\end{aligned}$     \\\hline          
\end{tabular}
\caption{For each annulus in a tetrahedron $\tau$, this declares which surface in $G$ of the same annulus type shall be used, under which permutation, and the transformed surface coordinates.}
\label{tab:surfCoordsTubings}
\end{table}

\end{document}